\documentclass[11pt]{amsart}
\usepackage{mathrsfs}
\usepackage{amsthm}
\usepackage{amsfonts}
\usepackage{latexsym,amsmath,amssymb}
\usepackage[colorlinks=true,linkcolor=blue]{hyperref}
\usepackage{indentfirst}

\renewcommand{\epsilon}{\varepsilon}

\newtheorem{theorem}{Theorem}[section]

\newtheorem{lemma}[theorem]{Lemma}

\newtheorem{Prop}[theorem]{Proposition}
\newtheorem{defn}[theorem]{Definition}

\newcommand{\bth}{\begin{theorem}}
	\newcommand{\ble}{\begin{lemma}}
		\newcommand{\bcor}{\begin{corr}}
			\newcommand{\bdeff}{\begin{defn}}
				\newcommand{\bprop}{\begin{proposition}}
					\newcommand{\ele}{\end{lemma}}
				\newcommand{\ecor}{\end{corr}}
			\newcommand{\edeff}{\end{defn}}
		
		\newcommand{\eprop}{\end{proposition}}

	\renewcommand{\Pi}{\varPi}

	\renewcommand{\epsilon}{\varepsilon}

	\numberwithin{equation}{section}

	\thanks{The first author was supported  by National Science Foundation of China(No.12101145) and  Guangxi Science, Technology Project (Grant No. GuikeAD22035202).}

\title
{Lagrangian Mean Curvature Flow in Pseudo-Euclidean Space \uppercase\expandafter{\romannumeral2}
}
\author{Shanshan Li}
\address{School of Mathematics and Statistics, Guangxi Normal University,
	Guilin, Guangxi 541004, People's Republic of China,
	E-mail: 2437826965@qq.com}
\author{Jiaru Lv}
\address{School of Mathematics and Statistics, Guangxi Normal University,
	Guilin, Guangxi 541004, People's Republic of China,
	E-mail: 1904363078@qq.com}
\author{Rongli Huang}
\address{School of Mathematics and Statistics, Guangxi Normal University,
	Guilin, Guangxi 541004, People's Republic of China,
	E-mail: ronglihuangmath@gxnu.edu.cn}
\begin{document}
	\maketitle
\begin{abstract}
		In this paper, we consider the mean curvature flow of entire Lagrangian graphs with initial data in the pseudo-Euclidean space, which is related to the special Lagrangian parabolic equation. We show that the parabolic equation \eqref{11} has a smooth solution $u(x,t)$ for three corresponding nonlinear equations between the Monge-Amp$\grave{e}$re type equation($\tau=0$) and the special Lagrangian parabolic equation($\tau=\frac{\pi}{2}$). Furthermore, we get the bound of $D^lu$, $l=\{3,4,5,\cdots\}$ for $\tau=\frac{\pi}{4}$ and the decay estimates of the higher order derivatives when $0<\tau<\frac{\pi}{4}$ and $\frac{\pi}{4}<\tau<\frac{\pi}{2}$. We also prove that $u(x,t)$ converges to smooth self-expanding solutions of \eqref{12}.
	\end{abstract}	
		
	\let\thefootnote\relax

\section{Introduction}
The mean curvature flow in higher codimension was studied extensively in the last few years. Jingyi Chen and Jiayu Li \cite{chen2004singularity} studied the tangent cones at first time singularity of a Lagrangian mean curvature flow and later Andr$\acute{e}$ Neves \cite{neves2013finite} developed a finite time singularity for Lagrangian mean curvature flow, which is in a four-dimensional compact Calabi-Yau manifold. Moreover, some authors studied the long-time existence and  convergence of mean curvature flow. In \cite{han2005mean}, Xiaoli Han and Jiayu Li showed that under suitable conditions which are concerned with the parabolic density of the mean curvature flow, the mean curvature flow with initial surface $\Sigma_0$ exists globally and converges to a holomorphic curve. And Mu-Tao Wang \cite{wang2002long} proved the long-time existence and convergence of graphic mean curvature flow in arbitrary codimension. Further, in \cite{smoczyk2004longtime}\cite{smoczyk2002mean}, Kunt Smoczyk and Smoczyk-Wang established long-time existence and convergence of Lagrangian mean curvature flow into a flat space under some convexity conditions respectively. Albert Chau, Jingyi Chen and Weiyong He\cite{chau2012lagrangian} had proved that the mean curvature flow of entire Lagrangian graphs with Lipschitz continuous initial data has the long time existence and convergence.  And Haozhao Li \cite{li2009convergence} extended the convergence of the Lagrangian mean curvature flow to a general K$\ddot{a}$hler-Einstein manifold under some stability conditions where dimension is arbitrary. And from \cite{xin2008mean}, Yuanlong Xin produced the mean curvature flow exists for all time with initial data $M_0$, if $M_0$ is complete, curvature is bounded and the image of the Gauss map is suitably small. Further, the evolution of mean curvature flow of a complete space-like hypersurface in a pseudo-Euclidean space manifold characterized by $(m,n)$ were studied by Yuanlong Xin \cite{xin2011mean}.  The priori estimate of a smooth solution $u(x,t)$ of Lagrangian mean curvature flow was obtained in dimensions $n\ge4$, that is $[u_t]_{1,\frac{1}{2};Q_{\frac{1}{2}}}+[D^2u]_{1,\frac{1}{2};Q_{\frac{1}{2}}}\le C(\parallel D^2u\parallel_{L^{\infty}(Q_1)})$, it follows from Tu A. Nguyen and Yu Yuan\cite{nguyen2011priori}. In the case of the Hessian of the potential is bounded, Albert Chau, Jingyi Chen and Yu Yuan\cite{chau2012rigidity} produced any space-like entire graphic self-shrinking solution of the Lagrangian mean curvature folw is flat when it is in $\mathbb{C}^n$ with the pseudo-Euclidean metric. The uniqueness of Lagrangian self-expanders in $\mathbb{C}^n$ with $n\ge2$ first introduced by Jason D Lotay and Andr$\acute{e}$ Neves in \cite{lotay2013uniqueness}. In this paper, we extend the Lagrangian mean curvature flow in the pseudo-Euclidean space for $0<\tau<\frac{\pi}{2}$ (cf. \cite{xin2018minimal}).\\
\indent The special Lagrangian parabolic equation can be written as
\begin{align}\label{11}
\begin{cases}
	\frac{\partial u}{\partial t}=F_{\tau}(D^{2}u), \quad t>0,x\in \mathbb{R}^{n},\\
\;\; u=u_{0}(x), \qquad\; t=0,x\in \mathbb{R}^{n}.
\end{cases}
\end{align}
with $\tau\in[0,\frac{\pi}{2}]$, $\lambda(D^{2}u)=(\lambda_{1},\lambda_{2},\cdots,\lambda_{n})$ are the eigenvalues of $D^{2}u$, where
 \begin{align*}
	F_{\tau}(\lambda)=	\left\{
	\begin{aligned}
		&\frac{1}{2}\sum_{i=1}^{n}\rm{ln}\lambda_{i},\quad &\tau=0,\\
		&\frac{\sqrt{a^{2}+1}}{2b}\sum_{i=1}^{n}\rm{ln}\frac{\lambda_{i}+a-b}{\lambda_{i}+a+b},\quad &0<\tau<\frac{\pi}{4},\\
		&-\sqrt{2}\sum_{i=1}^{n}\frac{1}{1+\lambda_{i}},\quad &\tau=\frac{\pi}{4},\\
		&\frac{\sqrt{a^{2}+1}}{b}\sum_{i=1}^{n}\rm{arctan}\frac{\lambda_{i}+a-b}{\lambda_{i}+a+b},\quad &\frac{\pi}{4}<\tau<\frac{\pi}{2},\\
		&\sum_{i=1}^{n}\rm{arctan}\lambda_{i},\quad &\tau=\frac{\pi}{2},
	\end{aligned}
	\right.
\end{align*}
$a=cot\tau$, $b=\sqrt{|cot^{2}\tau-1|}$.\\
By Proposition \ref{prop2.1}, there exists a family of diffeomorphisms
$$\varphi_{t}:\mathbb{R}^{n}\to\mathbb{R}^{n},$$
such that
$$Y(x,t)=(\varphi_{t},Du(\varphi_{t},t))\subset\mathbb{R}^{2n}_{n}$$
is a solution to the mean curvature flow in the pseudo-Euclidean space
\begin{align}\label{12}
\begin{cases}
		\frac{dY(x,t)}{dt}\;=\vec{H},\\
		Y(x,0)=Y_{0}(x).
\end{cases}
\end{align}
Here $\vec{H}$ is the mean curvature vector of the space-like submanifold $Y(x, t)\subset \mathbb{R}^{2n}_{n}$ and
$$Y_{0}(x)=(x,Du_{0}(x)).$$
\indent Next, we consider the long-time existence and convergence of the equation \eqref{11}. And the long-time existence and convergence results for solution of \eqref{11} have been obtained for $\tau=0$ and $\tau=\frac{\pi}{2}$ as follows.\\
\indent For $\tau=0$, \eqref{11} is Monge-Amp$\grave{e}$re type equation, Rongli Huang \cite{huang2011lagrangian} proved the long-time existence of solution which is smooth and got the bound of $D^lu$, $l=\{3,4,5,\cdots\}$ when Hessian satisfies $\mu I\le D^{2}u_{0}(x)\le\Lambda I$ where $\mu$, $\Lambda$ are positive constants, $\mu<\Lambda$ and $I$ is the identity matrix. Further, Rongli Huang and Zhizhang Wang \cite{huang2011entire} established uniform decay estimates ${\sup \limits_ {x\in \mathbb{R}^{n}}}|D^{3}u(\cdot,t)|\le\frac{C}{t}$ and ${\sup \limits_ {x\in \mathbb{R}^{n}}}|D^{l}u(\cdot,t)|\le\frac{C}{t^{l-2}}$ for $l\ge3$ away from time $t=0$.  \\
\indent For $\tau=\frac{\pi}{2}$, \eqref{11} is the special Lagrangian parabolic equation. Albert Chau, Jingyi Chen and Yu Yuan \cite{chau2013lagrangian} obtained it has a longtime smooth solution $u(x,t)$ for all $t>0$ and got the decay estimate ${\sup \limits_ {x\in \mathbb{R}^{n}}}|D^{l}u(x,t)|^2\le \frac{C_{l,\delta}}{t^{l-2}}$ for all $l\ge3$, when Hessian satisfies $-(1+\eta)I\le D^2u_0\le(1+\eta)I$ for a small positive dimensional constant $\eta=\eta(n)$.\\
\indent The main goal of this paper is to settle the long-time existence and convergence results of solution for the remaining cases of \eqref{11}, i.e., $\tau=\frac{\pi}{4}$, $0<\tau<\frac{\pi}{4}$ and $\frac{\pi}{4}<\tau<\frac{\pi}{2}$, respectively.
\begin{defn}
	Assume that $u_{0}(x)\in C^{2}(\mathbb{R}^{n})$. We call $u_{0}(x)$ satisfying\\
	\rm{(1)(Condition $A$)} $if$
	$$u_{0}(x)=\frac{u_{0}(Rx)}{R^{2}},\qquad\forall R>0;$$
	\rm{(2)(Condition $B$)} $if$
	$$(-1+\zeta)I\le D^{2}u_{0}(x)\le\varrho I,\quad x\in\mathbb{R}^{n},$$
	where $0<\zeta<\varrho+1$, $\tau=\frac{\pi}{4}$, $\varrho$ is a positive constant and $I$ is the identity matrix.\\
	\rm{(3)(Condition $E$)} $if$
	$$(b\mu-a)I\le D^{2}u_{0}(x)\le(b\Lambda-a)I,\quad x\in\mathbb{R}^{n},$$
	where $a=cot\tau$, $b=\sqrt{|cot^{2}\tau-1|}$, $0<\tau<\frac{\pi}{4}$, $\mu$, $\Lambda$ are positive constants, $\mu<\Lambda$ and $I$ is the identity matrix.\\
	\rm{(4)(Condition $L$)} $if$
	$$-(b+b\eta+a)I\le D^{2}u_{0}(x)\le (b+b\eta-a)I,\quad x\in\mathbb{R}^{n},$$
	where $\eta=\eta(n)$ is positive dimensional constant, $a=cot\tau$, $b=\sqrt{|cot^{2}\tau-1|}$, $\frac{\pi}{4}<\tau<\frac{\pi}{2}$ and $I$ is the identity matrix.
\end{defn}
 Our results are the followings:
\begin{theorem}\label{1.2}
	Let $u_{0} : \mathbb{R}^{n}\to\mathbb{R}$ be a $C^{2}$ function satisfying condition $B$ and consider the equation
	\begin{align}\label{121}
	\begin{cases}
	\frac{\partial u}{\partial t}=F_{\frac{\pi}{4}}(D^2u), \quad t>0,x\in \mathbb{R}^{n},\\
	\;\; u=u_{0}(x), \qquad\;\; t=0,x\in \mathbb{R}^{n}.
	\end{cases}
	\end{align}
	 Then there exists a unique solution of \eqref{121} such that
	\begin{align}\label{13}
	u(x,t)\in C^{\infty}(\mathbb{R}^{n}\times(0,+\infty)\cap C(\mathbb{R}^{n}\times[0,+\infty)).
	\end{align}
	where $u(\cdot,t)$ satisfies condition $B$. More generally, for $l=\{3,4,5,\cdots\}$ and $\varepsilon_{0}>0$, there holds
	\begin{align}\label{100005}
	{\sup \limits_ {x\in \mathbb{R}^{n}}}|D^{l}u(x,t)|\le C,\quad \forall t\in (\varepsilon_{0},+\infty),
	\end{align}
	where $C$ depends only on $n$, $\zeta$, $\varrho$, $\frac{1}{\varepsilon_{0}}.$
\end{theorem}
\begin{theorem}\label{1.21}
	Let $u_{0} : \mathbb{R}^{n}\to\mathbb{R}$ be a $C^{2}$ function satisfying condition $E$ and consider the equation
	\begin{align}\label{1211}
	\begin{cases}
	\frac{\partial u}{\partial t}=F_{\tau}(D^{2}u), \quad t>0,x\in \mathbb{R}^{n},\\
	\;\; u=u_{0}(x), \quad\quad\; t=0,x\in \mathbb{R}^{n}.
	\end{cases}
	\end{align}
	Then there exists
	a unique solution of \eqref{1211} when $0<\tau<\frac{\pi}{4}$ such that
	\begin{align}\label{1311}
	u(x,t)\in C^{\infty}(\mathbb{R}^{n}\times(0,+\infty)\cap C(\mathbb{R}^{n}\times[0,+\infty)).
	\end{align}
	where $u(\cdot,t)$ satisfies condition $E$. And then there exists a constant $C$ depending only on $n,a,b,\mu,\Lambda,\frac{1}{\varepsilon_{0}}$ such that
	\begin{align}
	{\sup \limits_ {x\in \mathbb{R}^{n}}}|D^{3}u(\cdot,t)|\le\frac{C}{t},\quad \forall t\in (\varepsilon_{0},+\infty),
	\end{align}
	More generally, for all $l=\{3,4,5,\cdots\}$ there holds
	\begin{align}
	{\sup \limits_ {x\in \mathbb{R}^{n}}}|D^{l}u(\cdot,t)|\le\frac{C}{t^{l-2}},\quad \forall t\in (\varepsilon_{0},+\infty),
	\end{align}
\end{theorem}
\begin{theorem}\label{1.112}
	Let $u_{0} : \mathbb{R}^{n}\to\mathbb{R}$ be a $C^{2}$ function satisfying condition $L$ and consider the equation
	\begin{align}\label{12111}
	\begin{cases}
	\frac{\partial u}{\partial t}=F_{\tau}(D^{2}u) \quad t>0,x\in \mathbb{R}^{n},\\
	\;\; u=u_{0}(x), \quad\;\;\; t=0,x\in \mathbb{R}^{n}.
	\end{cases}
	\end{align} Then there exists a unique solution of \eqref{12111} when $\frac{\pi}{4}<\tau<\frac{\pi}{2}$ such that
	\begin{align}\label{141}
	u(x,t)\in C^{\infty}(\mathbb{R}^{n}\times(0,+\infty)\cap C(\mathbb{R}^{n}\times[0,+\infty)).
	\end{align}
	where $u(\cdot,t)$ satisfies condition $L$. More generally, for $l=\{3,4,5,\cdots\}$, there holds
	\begin{align}\label{142}
	{\sup \limits_ {x\in \mathbb{R}^{n}}}|D^{l}u(x,t)|^2\le \frac{C_l}{t^{l-2}},\quad \forall t\in (\varepsilon_{0},+\infty),
	\end{align}
	where $C$ depends only on $n,l,\eta=\eta(n),\frac{1}{\varepsilon_{0}}$.
\end{theorem}
For Theorem \ref{1.2}, the existence results are based in a prior estimates on $u$. P. L. Lions and M. Musiela \cite{lions2006convexity}
introduced a class of fully nonlinear parabolic equations where the convexity properties of the
solutions are preserved. So we are able to derive a positive lower bound and an upper bound for
the eigenvalues of $D^{2}u$. By the Krylov-Safonov Theorem, we obtain the $C^{\alpha}$ norm of $D^{2}u$. But it seems difficult to get the bound of $D^{3}u$ only using interior Schauder estimates without the assumption of ${\sup \limits_{x\in \mathbb{R}^{n}}} |Du_{0}|< +\infty$. To overcome the difficulty, we will use the blow-up argument to prove
$${\sup \limits_{x\in \mathbb{R}^{n},\ t\ge\varepsilon_{0}}}|D^{3}u|<+\infty,$$
and further establish \eqref{100005} by interior Schauder estimates. Here we do not need the gradient
bound of $u_{0}$.\\
\indent Consider the following equation:
\begin{align}\label{15}
	F_{\tau}(D^{2}u)=u-\frac{1}{2}<x,Du>.
\end{align}
where $<\cdot,\cdot>$ denotes the standard inner product on $\mathbb{R}^{n}$.\\
\indent According to the definition in \cite{huang2022entire}, we can show that an entire solution to \eqref{15} is a self-expanding solution to Lagrangian mean curvature flow in the pseudo-Euclidean space.\\
\indent The following theorems show that we can obtain the self-expanding solutions by \eqref{11}.
\begin{theorem}\label{1.3}
	Let $u_{0}:\mathbb{R}^{n}\to\mathbb{R}$ be a $C^{2}$ function which satisfies condition $B$ and  consider the equation \eqref{121}. Assume that
	\begin{align}\label{16}
		\lim_{r\to+\infty}r^{-2}u_{0}(rx)=U_{0}(x),
	\end{align}
for some $U_{0}(x)\in C^{2}(\mathbb{R}^{n})$. Let $u(x, t)$ and $U(x, t)$ be two solutions to \eqref{121} with initial values $u_{0}(x)$ and $U_{0}(x)$ respectively. Then
\begin{align}
	\lim_{t\to+\infty}t^{-1}u(\sqrt{t}x,t)=U(x,1).
\end{align}\label{17}
Here the convergence is uniform and smooth in any compact subset of $\mathbb{R}^{n}$, and $U(x, 1)$ is a smooth self-expanding solution of \eqref{12}.
\end{theorem}
\begin{theorem}\label{1.312}
	Let $u_{0}:\mathbb{R}^{n}\to\mathbb{R}$ be a $C^{2}$ function which satisfies condition $E$ and  consider the equation \eqref{1211}. Assume that
	\begin{align}\label{1612}
	\lim_{r\to+\infty}r^{-2}u_{0}(rx)=U_{0}(x),
	\end{align}
	for some $U_{0}(x)\in C^{2}(\mathbb{R}^{n})$. Let $u(x, t)$ and $U(x, t)$ be two solutions to \eqref{1211} with initial values $u_{0}(x)$ and $U_{0}(x)$ respectively. Then
	\begin{align}
	\lim_{t\to+\infty}t^{-1}u(\sqrt{t}x,t)=U(x,1).
	\end{align}\label{1712}
	Here the convergence is uniform and smooth in any compact subset of $\mathbb{R}^{n}$, and $U(x, 1)$ is a smooth self-expanding solution of \eqref{12}.
\end{theorem}
\begin{theorem}\label{1.31}
	Let $u_{0}:\mathbb{R}^{n}\to\mathbb{R}$ be a $C^{2}$ function and the Hessian metrix satisfies condition $L$ and consider the equation \eqref{12111}. Assume that
	\begin{align}\label{161}
	\lim_{r\to+\infty}r^{-2}u_{0}(rx)=U_{0}(x),
	\end{align}
	for some $U_{0}(x)\in C^{2}(\mathbb{R}^{n})$. Let $u(x, t)$ and $U(x, t)$ be two solutions to \eqref{12111} with initial values $u_{0}(x)$ and $U_{0}(x)$ respectively. Then
	\begin{align}
	\lim_{t\to+\infty}t^{-1}u(\sqrt{t}x,t)=U(x,1).
	\end{align}\label{171}
	Here the convergence is uniform and smooth in any compact subset of $\mathbb{R}^{n}$, and $U(x, 1)$ is a smooth self-expanding solution of \eqref{12}.
\end{theorem}
This paper is organized as follows. In Sect. \ref{2}, we show that the mean curvature flow \eqref{12} can be derived by the special Lagrangian parabolic equation \eqref{11}  and give the definition of self-expanding solution to \eqref{11}. In sect. \ref{3}, Theorem \ref{1.2} is proved, by Theorem \ref{1.2}, we can present the proof of Theorem \ref{1.3}. In Sect. \ref{6}, we will show that Theorem \ref{1.21} and Theorem \ref{1.312}. The proof of Theorem \ref{1.112} and \ref{1.31} are given in Sect. \ref{4}.
\section{lagrangian self-expanding solution in $(\mathbb{R}^{2n},g_\tau)$}\label{2}
Throughout the following Einstein’s convention of summation over repeated indices will be adopted (see \cite{huang2022entire}).\\
\indent Let $(x^{1},\cdots,x^{n};y^{1},\cdots,y^{n})$ be null coordinates in $\mathbb{R}^{2n}_{n}$. Let $g_{\tau}=\sin\tau\delta_{0}+\cos\tau g_{0}$ be the linear combination of the standard Euclidean metric
$$\delta_{0}=\sum_{i=1}^{n}dx_{i}\otimes dx_{i}+\sum_{j=1}^{n}dy_{j}\otimes dy_{j},$$
and the pseudo-Euclidean metric
$$g_{0}=\sum_{i=1}^{n}dx_{i}\otimes dy_{i}+\sum_{j=1}^{n}dy_{j}\otimes dx_{j}.$$
We denote, for a smooth function $u$
$$u_{i}=\frac{\partial u}{\partial x_{i}},\quad u_{ij}=\frac{\partial^{2} u}{\partial x_{i}\partial x_{j}},\quad u_{ijk}=\frac{\partial^{3} u}{\partial x_{i}\partial x_{j}\partial x_{j}},\quad \cdots.$$
  \begin{Prop}\label{prop2.1}
  	If \eqref{11} admit an admissible solution $u(x,t)$ on $\mathbb{R}^{n}\times(0,T)$, then there exists a family of diffeomorphisms
  	$\varphi_{t}:\mathbb{R}^{n}\to\mathbb{R}^{n}$ for $t\in(0,T)$ such that $Y(x,t)=(\varphi(x,t),\nabla u(\varphi(x,t),t))$ solves \eqref{12} on
  	$\mathbb{R}^{n} \times(0,T)$.
\end{Prop}
    \begin{proof}
    Let $e_{i}=(0,\cdots,1,\cdots,0)$ be the $i$-th coordinate vector in $\mathbb{R}^{2n},i=1,\cdots,2n$. Denote the inner product of $(\mathbb{R}^{2n},g_{\tau})$ by $<\cdot,\cdot>_{\tau}$. The tangent vector fields of $M_{t}=\{(x,Du(x,t))|x\in\mathbb{R}^{n}\}$ are spanned by
    $$E_{i}=e_{i}+u_{ij}e_{n+j},\quad i=1,\cdots,n.$$
    The induced metric $g$ on $M_{t}$ is
    \begin{align}
    		\begin{aligned}
    	g_{ij}&=<E_{i},E_{j}>_{\tau}\\
    	&=<e_{i}+u_{ik}e_{n+k},e_{j}+u_{jl}e_{n+l}>_{\tau}\\
    	&=\sin\tau(\delta_{ij}+u_{ik}u_{kj})+2\cos\tau u_{ij}.
    		\end{aligned}
    \end{align}
    Denote $(g_{ij})^{-1}$ = $g^{ij}$. It is not hard to see that
    $$g^{ij}=\frac{\partial F_{\tau}(D^{2}u)}{\partial u_{ij}}.$$
    Let $\overline{\nabla}$ denote the Levi–Civita connection of $g_{\tau}$. We have $\overline{\nabla}_{E_{j}}E_{i}=u_{ijk}e_{n+k}$. The mean curvature of $M_{t}$ is computed as
    \begin{align}
    		\begin{aligned}
    		\vec{H}&=g^{ij}(\overline{\nabla}_{E_{i}}E_{j})^{\perp}\\
    		&=(g^{ij}u_{ijk}e_{n+k})^{\perp}\\
    		&=(\partial_{k}F_{\tau}(D^{2}u)e_{n+k})^{\perp}.
    	\end{aligned}
    \end{align}
    By the evolution equation \eqref{11} of $u$,
    $$\vec{H}=(u_{tk}e_{n+k})^{\perp}=(0,Du_{t})^{\perp}.$$
    Take a family of diffeomorphisms $\varphi_{t}:\mathbb{R}^{n}\times [0,T)\to\mathbb{R}^{n}$ that satisfies
    \begin{align*}
    (I_{n},D^{2}u(\varphi_{t},t))\cdot\frac{d\varphi_{t}}{dt}=-(0,Du_{t}(\varphi_{t},t))^{\top},
    \end{align*}
    where $\cdot$ denotes matrix product and $\perp$ denotes the projection to the tangent bundle of $M_{t}$. Set $Y=(\varphi_{t},Du(\varphi_{t},t))$. It follows that
    \begin{align*}
    		\begin{aligned}
    	\frac{dY}{dt}&=\left(\frac{d\varphi_{t}}{dt},Du_{t}(\varphi_{t},t)+D^{2}u(\varphi_{t},t) \cdot \frac{d\varphi_{t}}{dt}\right)\\
    	&=(0,Du_{t}(\varphi_{t},t))^{\perp}\\
    	&=\vec{H}(Y)
    	\end{aligned}
    \end{align*}
    This completes the proof.
    \end{proof}
\begin{defn}
	We say $u(x,t)$ is a self-expanding solution to \eqref{11}, if $u(x,t)$ is a solution that
	satisfies
	\begin{align}\label{23}
		u(x,t)=tu\left(\frac{x}{\sqrt{t}},1\right),\quad \forall t>0.
	\end{align}
	\end{defn}
	\begin{Prop}
		If $u(x,t)$ is a self-expanding solution to \eqref{11}, then $u(x,1)$ satisfies \eqref{15}.
	\end{Prop}
\begin{proof} By \eqref{11} and \eqref{23}, one can easily check that
\begin{align}
\begin{aligned}
F_{\tau}(D^{2}u(x,t))&=\frac{\partial u(x,t)}{\partial t}\\
&=u\left(\frac{x}{\sqrt{t}},1\right)-\frac{1}{2}\left<Du\left(\frac{x}{\sqrt{t}},1\right),\frac{x}{\sqrt{t}}\right>.
\end{aligned}
\end{align}
Taking $t=1$, we see that $u(x,1)$ satisfies \eqref{15}.
\end{proof}
\section{proof of Theorem\ref{1.2} and \ref{1.3}}\label{3}
We want to use the method of continuity and finite approximation to prove the solvability
of \eqref{121}.
\begin{defn}
	Given any $T>0$, $R>0$, $0<\alpha<1$ and set
	$$B_{R}=\{x\mid|x|<R,x\in\mathbb{R}^{n}\},\quad B_{R,T}=\{x\mid|x|<R,x\in\mathbb{R}^{n}\}\times(0,T),$$
	$$PB_{R,T}=B_{R}\times\{t=0\}\cup\partial B_{R}\times(0,T).$$
	Let $\gamma\in[0,1]$. We say $u\in C^{5+\alpha,\frac{5+\alpha}{2}}(B_{R}\times(0,T))\cap C(\overline{B}_{R}\times[0,T))$ is a solution to \eqref{25} if $u$ satisfies
	\begin{align}\label{25}
		\begin{cases}
			\frac{\partial u}{\partial t}-\gamma F_{\frac{\pi}{4}}(D^{2}u)-(1-\gamma)\Delta u=0,\quad(x,t)\in B_{R,T},\\
			u=u_{0}(x), \quad\qquad\qquad\qquad\qquad\quad\;\;(x,t)\in PB_{R,T}.
		\end{cases}
	\end{align}
\end{defn}
Clearly, there exists a unique solution $u(x, t)$ which satisfies \eqref{25} with $\gamma=0$. Let
$$K=\{\gamma\in[0,1]:\eqref{25}\ has\ a\ solution\}.$$
 To prove that the classical solutions to \eqref{121} satisfy the property, we need the following lemma which is proved by P.L.Lions, M.Musiela (cf. Theorem 3.1 of \cite{lions2006convexity},).
 \begin{lemma}\label{2.5}
 	Let $u:B_{R,T}\to \mathbb{R}$ be a classical solution of a fully nonlinear equation of the form
 	$$\frac{\partial u}{\partial t}=F(D^{2}u),$$
 	where $F$ is a $C^{2}$ function defined on the cone $\Gamma$ of definite symmetric matrices, which is monotone increasing (that is, $F(A)\le (A + B)$ whenever $B$ is a positive definite matrix), and the function
 	$$F^{*}(A)=-F(A^{-1})$$
 	is concave on $\Gamma_{+}$ of positive definite symmetric matrices. If $D^{2}u\ge0$ for $(x,t)\in PB_{R,T}$, then $D^{2}u\ge0$ for $(x,t)\in B_{R,T}$.
 \end{lemma}
\indent The following Lemma can be deduced from the work (cf. \cite{huang2011lagrangian}).
 \begin{lemma}\label{3.31}
	Suppose that $u:B_{R,T}\to \mathbb{R}$ be a classical solution of a fully nonlinear equation of the form
	\begin{align*}
		\frac{\partial u}{\partial t}=F(D^{2}u)
	\end{align*}
	where $F$ satisfies the conditions in Lemma \ref{2.5} and $F$ is concave on the cone $\Gamma_{+}$. If $\mu I\le D^{2}u\le \Lambda I$ for $(x,t)\in PB_{R,T}$, where $\mu$ and $\Lambda$ are positive constants. Then $\mu I\le D^{2}u\le \Lambda I$ for $(x,t)\in B_{R,T}.$
\end{lemma}
\indent Next we show that
 \begin{lemma}\label{2.6}
 	Suppose that $u:B_{R,T}\to \mathbb{R}$ be a classical solution of a fully nonlinear equation of the form
 	\begin{align*}
 	\frac{\partial u}{\partial t}=F_{\frac{\pi}{4}}(D^2u)	
 	\end{align*}
  If $(-1+\zeta)I\le D^{2}u\le\varrho I$ for $(x,t)\in PB_{R,T}$, then $(-1+\zeta)I\le D^{2}u\le \varrho I$ for $(x,t)\in B_{R,T}.$
 \end{lemma}
\begin{proof}
Let $\bar{u}=u+\frac{1}{2}\vert x\vert^2$, then we have
\begin{align*}
D^2\bar{u}=D^2u+I
\end{align*}
It follows that $D^2\bar{u}>0$ for $(x,t)\in PB_{R,T}$, and $\bar{u}$ satisfies
\begin{align*}
\frac{\partial \bar{u}}{\partial t}=\overline{F}(D^2\bar{u})
\end{align*}
Define
\begin{align*}
\begin{aligned}
&F(\lambda_i)=-\sqrt{2}\sum_{i=1}^{n}\frac{1}{1+\lambda_{i}},\\
&\overline{F}(\bar{\lambda}_i)=-\sqrt{2}\sum_{i=1}^{n}\frac{1}{\bar{\lambda}_{i}}.
\end{aligned}
\end{align*}
and $\overline{F}(\bar{\lambda}_i)$ is concave on a cone $\Sigma=\{\bar{\lambda}_{1}>0,\bar{\lambda}_{2}>0,\cdots,\bar{\lambda}_{n}>0\}.$ So  $\overline{F}(D^2\bar{u})$ is concave on $\Gamma_+$. Thus according to Lemma \ref{3.31}, we have $\mu I\le D^{2}\bar{u}\le \Lambda I$ for $(x,t)\in B_{R,T}.$ That is $\mu I\le D^{2}u+I\le \Lambda I$ for $(x,t)\in B_{R,T}.$ Further, we get $(\mu-1) I\le D^{2}u\le (\Lambda-1) I$ for $(x,t)\in B_{R,T}.$\\
Next, we take $\mu=\zeta$, $\Lambda=\varrho+1$, so $(-1+\zeta)I\le D^{2}u\le \varrho I$ for $(x,t)\in B_{R,T}.$
Thus Lemma \ref{2.6} is proved.
\end{proof}
 Given $x_{0}\in\mathbb{R}^{n}$, $\kappa>0$, define
 $$Q_{R,x_{0}}=\{x\mid|x-x_{0}|\le R\}\times[\kappa,\kappa+R),\quad Q_{\frac{R}{2},x_{0}}=\left\{x\mid|x-x_{0}|\le \frac{R}{2}\right\}\times\left[\kappa+\frac{R}{4},\kappa+\frac{R}{2}\right),$$
 $$Q_{\frac{R}{3},x_{0}}=\left\{x\mid|x-x_{0}|\le \frac{R}{3}\right\}\times\left[\kappa+\frac{R}{3},\kappa+\frac{5R}{12}\right),\quad B_{R,x_{0}}=\{x\mid|x-x_{0}|\le R\}.$$
 \indent The following two lemmas which will be mentioned below may be used repeatedly (cf. \cite{lieberman1996second}).\\
 \begin{lemma}[\label{2.7}cf. Lemma 14.6 of \cite{lieberman1996second}]
 Let $u:\mathbb{R}^{n}\times[0,T)\to \mathbb{R}$ be a classical solution of a fully nonlinear equation of the form
 \begin{align*}
 \begin{cases}
 \frac{\partial u}{\partial t}=F(D^{2}u), \quad t>0,x\in \mathbb{R}^{n},\\
\;\;u=u_{0}(x), \qquad t=0,x\in \mathbb{R}^{n}.
\end{cases}
 \end{align*}
where $F$ is a $C^{2}$ concave function defined on the cone $\Gamma$ of definite symmetric matrices, which is monotone increasing with
$$\mu I\le\frac{\partial F}{r_{ij}}\le \Lambda I.$$
Then there exists $0<\alpha<1$ such that
$$[D^{2}u]_{C^{\alpha,\frac{\alpha}{2}\left(\overline{Q}_{\frac{1}{2},x_{0}}\right)}}\le C_{1}|D^{2}u|_{C^{0}(\overline{Q}_{1,x_{0}})},$$
where $\alpha$, $C_{1}$ are positive constants depending only on $n$, $\mu$, $\Lambda$, $\frac{1}{\kappa}$.
 \end{lemma}
 \begin{lemma}[\label{2.8}cf. Theorem 4.9 of \cite{lieberman1996second}]
 Let $v:\mathbb{R}^{n}\times[0,T)\to \mathbb{R}$ be a classical solution of a linear parabolic equation of the form
 \begin{align*}
 	\begin{cases}
 		\frac{\partial v}{\partial t}-a^{ij}v_{ij}=0,\quad t>0,x\in \mathbb{R}^{n},\\
 		v=v_{0}(x), \qquad\quad \;t=0,x\in \mathbb{R}^{n}.
 	\end{cases}
 \end{align*}
 where there exists a positive constant $C_{2}$ such that
 $$\mu I\le a^{ij}\le \Lambda I,\quad[a^{ij}]_{C^{\alpha}\left(\overline{Q}_{\frac{R}{2},x_{0}}\right)}\le C_{2}.$$
 Then there holds
 $$|Dv|_{C^{0}\left(\overline{Q}_{\frac{R}{3},x_{0}}\right)}+|D^{2}v|_{C^{0}\left(\overline{Q}_{\frac{R}{3},x_{0}}\right)}+[D^{2}v]_{C^{\alpha,\frac{\alpha}{2}\left(\overline{Q}_{\frac{R}{3},x_{0}}\right)}}\le C_{3}|v|_{C^{0}(\overline{Q}_{R,x_{0}})},$$
 where $C_{3}$ is a positive constants depending only on $n$, $\mu$, $\Lambda$, $C_{2}$, $R$, $\frac{1}{\kappa}$.
 \end{lemma}
 According to problem \eqref{25}, we have the following results.
 \begin{lemma}\label{2.9}
 	$\overline{F}(\bar{\lambda})$ is a concave function, where
 	\begin{align*}
 \overline{F}(\bar{\lambda}_i)=-\sqrt{2}\sum_{i=1}^{n}\frac{1}{\bar{\lambda}_{i}}
 	\end{align*}
 	and $\bar{\lambda}\in \Sigma=\{\bar{\lambda}_{1}>0,\bar{\lambda}_{2}>0,\cdots,\bar{\lambda}_{n}>0\}.$
 \end{lemma}
 \begin{proof}
 	By calculating, we get the first derivative of $\overline{F}(\bar{\lambda})$.
 	\begin{align*}
 	\frac{\partial \overline{F}(\bar{\lambda})}{\partial \bar{\lambda}_{i}}=
 	\sqrt{2}\frac{1}{\bar{\lambda}_i^{2}}
 	\end{align*}
 	Further, considering the Hessian matrix $\dfrac{\partial^{2} \overline{F}(\bar{\lambda})}{\partial \bar{\lambda}_{i}\partial \bar{\lambda}_{j}}$, it follows that $\dfrac{\partial^{2} \overline{F}(\bar{\lambda})}{\partial \bar{\lambda}_{i}\partial \bar{\lambda}_{j}}$=0 for $i\neq j$, if $i=j$, $\dfrac{\partial^{2}\overline{F}(\bar{\lambda})}{\partial \bar{\lambda}_{i}\partial \bar{\lambda}_{j}}=-2\sqrt{2}\dfrac{1}{\bar{\lambda}_i^{3}}$. Thus the Hessian matrix  $\dfrac{\partial^{2} \overline{F}(\bar{\lambda})}{\partial \bar{\lambda}_{i}\partial \bar{\lambda}_{j}}$ is a diagonal matrix with the diagonal elements being $-2\sqrt{2}\dfrac{1}{\bar{\lambda}_i^{3}}$ and the remaining elements being 0. Specifically, the Hessian matrix can be expressed as:
 	\begin{align*}
 	\begin{bmatrix}
 	-2\sqrt{2}\dfrac{1}{\bar{\lambda}_i^{3}} & & &\\
 	& -2\sqrt{2}\dfrac{1}{\bar{\lambda}_i^{3}} & &\\
 	& & \ddots &\\
 	& & & -2\sqrt{2}\dfrac{1}{\bar{\lambda}_i^{3}}
 	\end{bmatrix}
 	\end{align*}
 	Thus $\dfrac{\partial^{2} \overline{F}(\bar{\lambda})}{\partial \bar{\lambda}_{i}^{2}}$ is negative in a cone $\Sigma=\{\bar{\lambda}_{1}>0,\bar{\lambda}_{2}>0,\cdots,\bar{\lambda}_{n}>0\}$. Moreover $\overline{F}(\bar{\lambda})$ is a concave function.
 \end{proof}
 \begin{lemma}\label{2.10}
 $K=\{\gamma\in[0,1]:\eqref{25}\ has\ a\ solution\}$, then $K$ is closed.
 \end{lemma}
\begin{proof} Let $\bar{u}=u+\frac{1}{2}\vert x\vert^2$, then we have $D^2\bar{u}=D^2u+I$, note that $D^2\bar{u}>0$ for $\bar{\lambda}\in\Sigma=\{\bar{\lambda}_{1}>0,\bar{\lambda}_{2}>0,\cdots,\bar{\lambda}_{n}>0\}.$\\
Suppose that $u$ is a solution of the equation \eqref{25}
\begin{align*}
	\begin{cases}
\frac{\partial u}{\partial t}-\gamma F_{\frac{\pi}{4}}(D^{2}u)-(1-\gamma)\Delta u=0,\quad(x,t)\in B_{R,T},\\
u=u_{0}(x), \quad\qquad\qquad\qquad\qquad\quad\;\;\;\;(x,t)\in PB_{R,T}.
\end{cases}
\end{align*}
it is clear that
\begin{align}
	\begin{cases}
\frac{\partial \bar{u}}{\partial t}-\gamma \overline{F}(D^{2}\bar{u})-(1-\gamma)(\Delta \bar{u}-1)=0,\quad(x,t)\in B_{R,T},\\
\bar{u}=\bar{u}_{0}(x), \quad\qquad\qquad\qquad\qquad\quad\qquad\;\;\;\;(x,t)\in PB_{R,T}.
\end{cases}
\end{align}
 For $\bar{A}\in\Gamma_+$, set
 $$\overline{F}(\bar{A})=\gamma \overline{F}(\bar{\lambda} (\bar{A}))+(1-\gamma)\text{Tr}\bar{A}.$$
 Let $\bar{\lambda}(\bar{A})=(\bar{\lambda}_{1},\bar{\lambda}_{2},\cdots,\bar{\lambda}_{n})$ are the eigenvalues of $\bar{A}$. Define
 $$\overline{F}(\bar{A})=\gamma  \overline{F}(\bar{\lambda}(\bar{A}))+(1-\gamma)(\bar{\lambda}_{1}+\bar{\lambda}_{2}+\cdots+\bar{\lambda}_{n}),$$
 $$\overline{F}^{*}(\bar{A})=\gamma \overline{F}(\bar{\lambda}(\bar{A}))-(1-\gamma)\left(\frac{1}{\bar{\lambda}_{1}}+\frac{1}{\bar{\lambda}_{2}}+\cdots+\frac{1}{\bar{\lambda}_{n}}\right),$$
 By Lemma \ref{2.9}, $D^{2}\overline{F}(\bar{\lambda})$ is negative in a cone $\Sigma=\{\bar{\lambda}_{1}>0,\bar{\lambda}_{2}>0,\cdots,\bar{\lambda}_{n}>0\}$, thus $\overline{F}$, $\overline{F}^{*}$ are smooth concave functions defined on the cone $\Gamma_+$, which are monotone increasing.\\
	\indent It follows from Lemma \ref{3.31} that if $\bar{u}_{0}(x)$ satisfies $\mu I\le D^2\bar{u}_0\le\Lambda I$, then $\bar{u}(x,t)$ does so. For $T>s>0$, $R>\epsilon>0$, define
 $$B_{R-\epsilon,T}=B_{R-\epsilon}\times(0,T),\quad B_{R-\epsilon}(T,s)=B_{R-\epsilon}\times(s,T).$$
 Furthermore, combining Lemma \ref{2.7} and Lemma \ref{2.8}, we have
 \begin{align}\label{26}
 	||\bar{u}||_{C^{2,1}(\overline{B}_{R-\epsilon,T})}\le C_{1},\quad||\bar{u}||_{C^{2+\alpha,\frac{2+\alpha}{2}}(\overline{B}_{R-\epsilon}(T,s))}\le C_{2},
 \end{align}
 where $0<\alpha<1$, $C_{1}$ is a positive constant depending only on $u_{0}$, $R$, $T$ and $C_{2}$ relies on $u_{0},\mu,\Lambda,R,T,\frac{1}{\epsilon},\frac{1}{s}$.\\
 And so we obtain
 \begin{align}\label{99}
 	||u||_{C^{2,1}(\overline{B}_{R-\epsilon,T})}\le C_{3},\quad||u||_{C^{2+\alpha,\frac{2+\alpha}{2}}(\overline{B}_{R-\epsilon}(T,s))}\le C_{4}
 \end{align}
  By \eqref{99}, a diagonal sequence argument and the regularity theory of parabolic equations imply that $K$ is closed.
 \end{proof}
\indent To prove that $K$ is open we need the following lemma (cf. Theorem 17.6 of  \cite{gilbarg1977elliptic})
 \begin{lemma}\label{2.11}
 	Let $\mathcal{B}_{1}$, $\mathcal{B}_{2}$ and $\textbf{X}$ be Banach spaces and $G$ be a mapping from an open subset of $\mathcal{B}_{1}\times\textbf{X}$ into $\mathcal{B}_{2}$. Let $(\tilde{u},\tilde{\gamma})$ be a point in $\mathcal{B}_{1}\times\textbf{X}$ satisfying that\\
 	\rm{(1)}$\;$ $G(\tilde{u},\tilde{\gamma})$=0,\\
 	\rm{(2)}$\;$ $G$ in continuously differentiable at $(\tilde{u},\tilde{\gamma})$,\\
 	\rm{(3)}$\;$ the partial Fr$\rm{\acute{e}}$chet derivative $L=G^1_{(\tilde{u},\tilde{\gamma})}$ is invertible.\\
 	Then there exists a neighbourhood $\mathcal{N}$ of $\tilde{\gamma}$ in $\textbf{X}$ such that the eqaution $G(u,\gamma)=0$ is solvable for each $\gamma\in\mathcal{N}$ with $u=u_{\gamma}\in\mathcal{B}_1$.
 \end{lemma}
 By the implicit function theorem, we have the next lemma.
 \begin{lemma}\label{2.12}
$ K=\{\gamma\in[0,1]:\eqref{25}\ has\ a\ solution\}$, $K$ is open.
 \end{lemma}
\begin{proof} Define the Banach spaces
 $$\textbf{X}=\mathbb{R},$$
 $$\mathcal{B}_{1}=C^{2+\alpha,\frac{2+\alpha}{2}}(\overline{B}_{R,T}),$$
 $$\mathcal{B}_{2}=C^{\alpha,\frac{\alpha}{2}}(\overline{B}_{R,T})\times C^{2+\alpha,\frac{2+\alpha}{2}}(PB_{R,T})$$
 and a differentiable map from $\mathcal{B}_{1}\times\textbf{X}$ into $\mathcal{B}_{2}$,
 $$G:(u,\gamma)\to \left[\frac{\partial u}{\partial t}-(\gamma F_{\frac{\pi}{4}}(D^2u)+(1-\gamma)\Delta u),u-u_0\right].$$
 Similarly, suppose that $\bar{u}=u+\frac{1}{2}\vert x\vert^2$, then we have
 \begin{align*}
 \overline{G}:(\bar{u},\gamma)\to \left[\frac{\partial \bar{u}}{\partial t}-(\gamma \overline{F}(D^2\bar{u})+(1-\gamma)(\Delta \bar{u}-1)),\bar{u}-\bar{u}_0\right]
 \end{align*}
 We take an open set of $\mathcal{B}_{1}\times\textbf{X}$:
 $$\Theta=\left\{\bar{u}\bigg|\frac{\mu}{2}I<D^2\bar{u}(x,t)<\frac{3\Lambda}{2}I,\bar{u}\in\mathcal{B}_{1}\right\}\times(0,1).$$

 Suppose that $(\bar{u},\gamma)\in \Theta.$ Then the partial Fr$\rm{\acute{e}}$chet derivative $L=\overline{G}^1_{(\bar{u},\gamma)}$ is invertible if and only if the following Cauchy-Dirichlet problem is solvale:
 $$	
 \begin{cases}
 	\frac{\partial w}{\partial t}-\gamma \bar{g}^{ij}\frac{\partial^2w}{\partial x^i\partial x^j}-(1-\gamma)\Delta w=h, \quad(x,t)\in B_{R,T},\\
 	w=g, \quad\qquad\qquad\qquad\qquad\qquad\qquad(x,t)\in PB_{R,T}.
 \end{cases}
 $$
 where $(h,g)\in\mathcal{B}_{2}$ and $\bar{g}^{ij}=\dfrac{\partial \overline{F}(D^2\bar{u})}{\partial\bar{u}_{ij} }$. Using the theory of the linear parabolic equations (cf. \cite{lieberman1996second}) we can do it.\\
	\indent Therefore applying Lemma \ref{2.11} and using approximate methods, we get $\overline{G}_{(\bar{u},\gamma)}=0$ is solve. Further, we can solve $G_{(\bar{u},\gamma)}=0$. Thus we deduce that $K$ is open.
	\end{proof}
 \begin{lemma}\label{2.13}
 	Let $u_0:\mathbb{R}^n\to\mathbb{R}$ be a $C^2$ function which satisfies condition $B$. Then there exists a unique solution of \eqref{121} such that $u(x,t)$ satisfies condition $B$ and \eqref{13}
 \end{lemma}
 \begin{proof} For $N\in \textbf{Z}^+$, $T>0$, consider the Cauchy-Dirichlet problem
 \begin{align}\label{27}
 	\begin{cases}
 		\frac{\partial u}{\partial t}-F_{\frac{\pi}{4}}(D^{2}u)=0,\quad(x,t)\in B_{N,T},\\
 		u=u_{0}, \qquad\qquad\qquad(x,t)\in PB_{N,T}.
 	\end{cases}
 \end{align}
 By Lemma \ref{2.10} and \ref{2.12}, there exists a unique solution of \eqref{27}. We denote it by $u_{N}(x,t)$. Lemma \ref{2.6} tells us that $u_{N}(x,t)$ satisfies condition $B$. Suppose that $\bar{u}_N=u_N+\frac{1}{2}\vert x\vert^2$, for $Q_{R,x_{0}}\subset\subset B_{N,T}$, by Lemma \ref{2.7} and \ref{2.8}, there exists a positive constant $C$ independent of $N$ such that
 $$[D^{2}\bar{u}_{N}]_{C^{\alpha,\frac{\alpha}{2}}\left(\overline{Q}_{\frac{R}{3},x_{0}}\right)}\le C.$$
 Note that
 \begin{align*}
 [D^{2}u_{N}]_{C^{\alpha,\frac{\alpha}{2}}\left(\overline{Q}_{\frac{R}{3},x_{0}}\right)}\le C_5.
 \end{align*}
 Combining $\mu I\le D^2(\bar{u}_{N})_0\le\Lambda I$, there exists a positive constant $\widetilde{C}$ independent of $N$ and $\frac{1}{\kappa}$ such that
 $$|\bar{u}_{N}|_{C^{0}\left(\overline{Q}_{\frac{R}{3},x_{0}}\right)}+|D\bar{u}_{N}|_{C^{0}\left(\overline{Q}_{\frac{R}{3},x_{0}}\right)}+|D^{2}\bar{u}_{N}|_{C^{0}\left(\overline{Q}_{\frac{R}{3},x_{0}}\right)}\le \widetilde{C}.$$
 Therefore,
 \begin{align*}
 |u_{N}|_{C^{0}\left(\overline{Q}_{\frac{R}{3},x_{0}}\right)}+|Du_{N}|_{C^{0}\left(\overline{Q}_{\frac{R}{3},x_{0}}\right)}+|D^{2}u_{N}|_{C^{0}\left(\overline{Q}_{\frac{R}{3},x_{0}}\right)}\le \widetilde{C}_1.
 \end{align*}
 a diagonal sequence argument and the regularity theory of parabolic equations imply that we obtain the desired results.
 \end{proof}
 \begin{proof}[Proof of Theorem \ref{1.2}]Using Lemma \ref{2.13}, there exists a unique solution of \eqref{121} satisfying \eqref{13} and condition $B$.
 By Lemma \ref{2.7} we get
 \begin{align}\label{3811}
 	[D^{2}u]_{C^{\alpha,\frac{\alpha}{2}\left(\overline{Q}_{\frac{1}{2},x_{0}}\right)}}\le C,
 \end{align}
where $C$ is a positive constant depending on $n,\mu,\Lambda$ and $\frac{1}{\kappa}$.

We will derive higher order estimate \eqref{100005} via the blow up argument for $l=3$. To do so, By \cite{andrews2007pinching}, we employ a parabolic scaling now. The remaining proof is routine. Define
 $$y=\sigma(x-x_{0}),\quad s=\sigma(t-t_{0}),$$
 $$u_{\sigma}(y,s)=\sigma^{2}\left[(x,t)-u(x_{0},t_{0})-Du(x_{0},t_{0})\cdot(x-x_{0})\right].$$
 It is easy to see that
 $$D^{2}_{y}u_{\sigma}=D^{2}_{x}u,\quad\frac{\partial}{\partial s}u_{\sigma}=\frac{\partial}{\partial t}u$$
 and
 $$D_{y}^{l}u_{\sigma}=\sigma^{2-l}D^{l}_{x}u$$
 for all nonnegative integers $l$. By computing, $u_{\sigma}(y,s)$ satisfies
 	\begin{align*}
 		\begin{cases}
 			\frac{\partial u_{\sigma}}{\partial s}=F_{\frac{\pi}{4}}(D^{2}u_{\sigma}), \qquad s>0,y\in \mathbb{R}^{n},\\
 			\;\;u_{\sigma}=u_{\sigma}(y,s)|_{t=t_{0}}, \quad s=0,y\in \mathbb{R}^{n}.
 	\end{cases}
 	\end{align*}
 	with
 	\begin{align}\label{29}
 			u_{\sigma}(0,0)=Du_{\sigma}(0,0)=0.
 	\end{align}
 Suppose that $|D^{3}u|^{2}$ is not bounded in $\mathbb{R}^{n}\times\left[\epsilon_{0},+\infty\right)$. By Lemma 3.5 of \cite{huang2011blow}, there would be a sequence $\{t_{k}\}$($t_{k}\ge\epsilon_{0}$) and $R_{k}\to +\infty$, such that
 \begin{align}\label{210}
 	2\rho_{k}:={\sup \limits_{x\in B_{R_{k},x_{0}}}}|D^{3}u(x,t_{k})|^{2}\to+\infty
 \end{align}
 and
 \begin{align}\label{211}
 \sup_{\substack{x\in B_{R_{k}},x_{0}\\t\le t_{k}}}|D^{3}u(x,t)|^{2}\le 2\rho_{k}.
 \end{align}
 Then there exists $x_{k}$ such that
 \begin{align}\label{212}
 	 |D^{3}u(x_{k},t_{k})|^{2} \ge \rho_{k}\to\infty,\quad k\to+\infty.
 \end{align}
Let $(y,Du_{\sigma_{k}}(y,s))$ be a parabolic scaling of $(x,Du(x,t))$ by $\sigma_{k}=(\rho_{k})^{\frac{1}{2}}$ at $(x_{k},t_{k})$ for each $k$.\\
Thus $u_{\sigma_{k}}(y,s)$ is a solution of a fully nonlinear parabolic equation
\begin{align}\label{213}
	\frac{\partial u_{\sigma_{k}}}{\partial s}-F_{\frac{\pi}{4}}(D^{2}u_{\sigma_{k}})=0,\quad0<s\le\sigma^{2}_{k}t_{k},\ y\in\mathbb{R}^{n}.
\end{align}
Combining \eqref{210}-\eqref{212} and \eqref{3811}, we arrive at
\begin{align}\label{214}
	(-1+\zeta)I\le D^{2}_{y}u_{\sigma_{k}}=D^{2}_{x}u\le\varrho I,\quad(y,s)\in\mathbb{R}^{n}\times[0,+\infty);
\end{align}
for all $y_{1},y_{2}\in \mathbb{R}^{n}$, $y_{1}=u_{\sigma_{k}}(x_{1}-x_{0})$, $y_{2}=u_{\sigma_{k}}(x_{2}-x_{0})$,
$$\frac{|D^{2}_{y_{1}}u_{\sigma_{k}}-D^{2}_{y_{2}}u_{\sigma_{k}}|}{|y_{1}-y_{2}|^{\alpha}}=\sigma_{k}^{-\alpha}\frac{|D^{2}_{x_{1}}u-D^{2}_{x_{2}}u|}{|x_{1}-x_{2}|^{\alpha}}\le\sigma_{k}^{-\alpha}C\to0,$$
and
\begin{align}\label{215}
	\begin{aligned}
			&|D^{3}_{y}u_{\sigma_{k}}|^{2}=\sigma_{k}^{-2}|D^{3}_{x}u|^{2}\le2,\quad\forall y\in\mathbb{R}^{n},\\
		&|D^{3}_{y}u_{\sigma_{k}}(0,0)|\ge1.
	\end{aligned}
\end{align}
For each $i$, set $w=D_{x^{i}} u_{\sigma_k}$. From \eqref{213}, $w$ satisfies
$$\frac{\partial w}{\partial s}-g^{ij}_{\sigma_k}w_{ij}=0.$$
Using \eqref{215} and Lemma \ref{2.7}, there exists a constant $C$ depending only on $n,\mu,\Lambda,\frac{1}{\kappa_{0}}$, such that we derive
\begin{align}\label{216}
	[D^{3}_{y}u_{\sigma_{k}}]_{C^{\alpha,\frac{\alpha}{2}\left(\overline{Q}_{\frac{1}{2},y_{0}}\right)}}\le C,\quad \forall y\in\mathbb{R}^{n}.
\end{align}
Combining \eqref{29} and \eqref{214}-\eqref{216} together, a diagonal sequence argument shows that $u_{\sigma_{k}}$ converges subsequentially and uniformly on any compact subset in $\mathbb{R}^{n}\times[0,+\infty)$ to a smooth function $u_{\infty}$ with
$$[D^{2}_{y}u_{\infty}]_{C^{\alpha,\frac{\alpha}{2}(\overline{Q}_{\frac{1}{2},y_{0}})}}=0,\quad\forall (y,s)\in\mathbb{R}^{n}\times[0,+\infty)$$
and
$$|D^{3}_{y}u_{\infty}(0,0)|\ge1.$$
It is a contradiction. So ${\sup \limits_{x\in \mathbb{R}^{n},t\ge\epsilon_{0}}}|D^{3}u(x,t)|\le C$. From equation \eqref{11}, using the interior Schauder estimates, we obtain \eqref{100005} for $l=3,4,5,\cdots$.
\end{proof}
The following lemma shows that how the self-expanding solutions are constructed by the flow \eqref{121}.
\begin{lemma}\label{2.14}
	If $u_{0}$ satisfies condition $A$ and $B$. Then $u(x,1)$ is a smooth solution to \eqref{15}.
\end{lemma}
\begin{proof} The main idea comes from \cite{chau2009entire}, which we present here for completeness.
	If $u_{0}$ satisfies condition $A$ and $B$. Then by Theorem \ref{1.2}, there exists a unique smooth solution $u(x,t)$ to \eqref{121} for all $t>0$ with initial value $u_{0}$. One can verify that
	$$u_{R}(x,t):=R^{-2}u(Rx,R^{2}t)$$
	is a solution to \eqref{121} with initial value
	$$u_{R}(x,0):=R^{-2}u_{0}(Rx)=u_{0}(x).$$
	Here condition $A$ is used. Since $u_{R}(x,0)=u_{0}$, the uniqueness results in Theorem \ref{1.2} imply
	$$u(x,t)=u_{R}(x,t)$$
	for any $R>0$. Therefore $u(x,t)$ satisfies \eqref{23}, and hence $u(x,1)$ solves \eqref{15}. In other words, $u(x,1)$ is a smooth self-expanding solution.
	\end{proof}
	We present here the proof of Theorem \ref{1.3} by the methods of \cite{chau2009entire}.
\begin{proof}[Proof of Theorem \ref{1.3}]
 Assume that
 \begin{align*}
 	U_0(x)=\lim_{R\to+\infty}R^{-2}u_0(Rx).
 \end{align*}
	Thus $U(x,0)$ satisfies condition $B$ and we obtain
	$$U_0(x)=\lim_{R\to\infty}R^{-2}u_0(Rx)=\lim_{R\to\infty}R^{-2}l^{-2}u_0(Rlx)=l^{-2}U_0(lx),$$
	namely, $U_0(x)$ satisfies condition $A$. Then by Lemma \ref{2.14}, we conclude that $U(x,l)$ is a self-expanding solution.\\
	
	Define
	$$u_{R}(x,t):=R^{-2}u(Rx,R^{2}t).$$
	It is clear that $u_{R}(x,t)$ is a solution to \eqref{121} with initial value $u_{R}(x,0)=R^{-2}u_0(Rx)$ satisfying condition $B$.
	
	For any sequence $R_i\to+\infty$, we consider the limitation of $u_{R_i}(x,t)$. For $t>0$, there holds
	$$D^2u_{R_i}(x,t)=D^2u(R_ix,R^{2}_it).$$
	Using Theorem \ref{1.2}, we have
	$$(-1+\zeta) I\le D^2u_{R_i}(x,t)\le\varrho I.$$
	for all $x$ and $t>0$. Moreover, according to \eqref{100005} in Theorem \ref{1.2}, we get
	$$	{\sup \limits_ {x\in \mathbb{R}^{n}}}|D^{l}u_{R_i}(\cdot,t)|\le C,\quad \forall t\ge\varepsilon_{0},\ l={3,4,5\cdots}.$$
	For any $m\ge1$, $l\ge0$, using \eqref{121}, there exists a constant $C$ such that
	 	$$	{\sup \limits_ {x\in \mathbb{R}^{n}}}\bigg|\frac{\partial^m}{\partial t^m}D^{l}u_{R_i}\bigg|\le C,\quad \forall t\ge\varepsilon_{0},\ l={3,4,5\cdots}.$$
	 	We observe that
	 	$$u_{R_i}=R^{-2}_iu_0 \quad and \quad Du_{R_i}(0,0)=R^{-1}_iDu_0(0)$$
	 	are both bounded. Thus $u_{R_i}(0,t)$ and $Du_{R_i}(0,t)$ are uniformly bounded with respect to $i$ for any fixed $t$. By the Arzel$\rm{\grave{a}}$-Ascoli theorem, there exists a subsequence $\{R_{k_{i}}\}$ such that $u_{R_{k_{i}}}(x,t)$ converges uniformly to a solution $\widehat{U}(x,t)$ to \eqref{121} in any compact subsets of $\mathbb{R}^{n}\times(0,+\infty)$, and $\widehat{U}(x,t)$ satisfies the estimates in Theorem \ref{1.2}. Since $\frac{\partial\widehat{U}}{\partial t}$ is  uniformly bounded for any $t>0$, $\widehat{U}(x,t)$ converges to some function $\widehat{U}_0(x)$ when $t\to0$. One can verify that
	 	$$\begin{aligned}
	 		\widehat{U}_0(x)&=\lim_{t\to0}\widehat{U}(x,t)\\
	 		&=\lim_{t\to0}\lim_{i\to+\infty}R_i^{-2}u(R_ix,R^{2}_i t)\\
	 		&=\lim_{i\to+\infty}\lim_{t\to0}R_i^{-2}u(R_ix,R^{2}_i t)\\
	 		&=\lim_{i\to+\infty}R_i^{-2}u_0(R_ix)\\
	 		&=U_0(x).
	 	\end{aligned} $$
By the uniqueness results, the above limit is independent of the choice of the subsequence $\{R_{i}\}$ and
$$\widehat{U}(x,t)=U(x,t).$$
	\indent Thus, letting $R=\sqrt{t}$, we have $t^{-1}u(\sqrt{t}x,t)=u_{\sqrt{t}}(x,1)$ converging to $U(x,1)$ uniformly in compact subsets of $\mathbb{R}^{n}$ when $t\to+\infty.$ Theorem \ref{1.3} is established.
\end{proof}
\section{proof of Theorem \ref{1.21} and \ref{1.312}}\label{6}
In this section, we will show Theorem \ref{1.21} and \ref{1.312} by transformation. We take the following trivial transform
\begin{align}\label{41}
\varphi(x)=\frac{b}{\sqrt{a^{2}+1}}u\left(\frac{(a^2+1)^{\frac{1}{4}}}{b}x\right)+\frac{a}{2b}\vert x\vert^2
\end{align}
Then, $\varphi$ satisfies
\begin{align}\label{42}
	\begin{cases}
\frac{\partial \varphi}{\partial t}=\frac{1}{2}\sum_{i=1}^{n}\ln\nu_i, \quad t>0,x\in \mathbb{R}^{n},\\
\;\; \varphi=\varphi_{0}(x), \qquad\quad\;\; t=0,x\in \mathbb{R}^{n}.
\end{cases}
\end{align}
it follows that $\varphi$ satisfies Monge-Amp$\grave{e}$re type equation
\begin{align}\label{43}
	\begin{cases}
\frac{\partial \varphi}{\partial t}=\frac{1}{2}\ln\det D^2\varphi, \quad t>0,x\in \mathbb{R}^{n},\\
\;\; \varphi=\varphi_{0}(x), \quad\qquad\;\;\; t=0,x\in \mathbb{R}^{n}.
\end{cases}
\end{align}
where $\nu=(\nu_1,\nu_2,\cdots,\nu_n)$ are the eigenvalues of $D^2\varphi$.\\
\indent On the other hand, by \eqref{41}, we get
\begin{align*}
\varphi_0(x)=\frac{b}{\sqrt{a^{2}+1}}u_0\left(\frac{(a^2+1)^{\frac{1}{4}}}{b}x\right)+\frac{a}{2b}\vert x\vert^2
\end{align*}
we now derive the second derivative for \eqref{43} via computing
\begin{align*}
   D^2\varphi_0(x)=\frac{1}{b}D^2u_0\left(\frac{(a^2+1)^{\frac{1}{4}}}{b}x\right)+\frac{a}{b}I
\end{align*}
Since
\begin{align*}
(b\mu-a)I\le D^2u_0\le(b\Lambda-a)I
\end{align*}
we obtain
\begin{align*}
\left(\mu-\frac{a}{b}\right)I\le\frac{1}{b}D^2u_0\left(\frac{(a^2+1)^{\frac{1}{4}}}{b}x\right)\le\left(\Lambda-\frac{a}{b}\right)I
\end{align*}
it is clear that
\begin{align*}
\mu I\le D^2\varphi_0(x)\le\Lambda I
\end{align*}
Therefore the equation \eqref{43} has a long-time smooth solution $\varphi(x,t)$, and $\varphi(x,t)$ produces the decay estimate of the higher derivatives ${\sup \limits_ {x\in \mathbb{R}^{n}}}|D^{l}\varphi(x,t)|\le \frac{C_l}{t^{l-2}}$ for $l=\{3,4,5,\cdots\}$ according to Theorem 1.1 in \cite{huang2011lagrangian} and Theorem 1.3 from \cite{huang2011entire}, and so is $u$. Further, Theorem \ref{1.21} is proved. Similarly, from \cite{huang2011lagrangian}, we conclude that $\varphi$ satisfies $\lim_{t\to+\infty}t^{-1}\varphi(\sqrt{t}x,t)=U(x,1)$ where $U(x,1)$ is a smooth self-expanding solution of \eqref{12}. This completes the proof of Theorem \ref{1.312}.
\section{proof of Theorem \ref{1.112} and Theorem \ref{1.31}}\label{4}
We prove Theorem \ref{1.112} and Theorem \ref{1.31} by transformation. We are eager to convert $F_{\tau}(\lambda)= \frac{\sqrt{a^{2}+1}}{b}\sum_{i=1}^{n}\rm{arctan}\frac{\lambda_{i}+a-b}{\lambda_{i}+a+b}$ into $F_{\tau}(\lambda)=\sum_{i=1}^{n}\rm{arctan}\lambda_{i}.$
\begin{proof}[proof of Theorem \ref{1.112}]
	For $\frac{\pi}{4}<\tau<\frac{\pi}{2}$, the Hessian matrix satisfies $-(b+b\eta+a)I\le D^2u_0\le(b+b\eta-a)I$. By the difference formula for tangent,
	\begin{align}\label{510}
	\rm{arctan}\frac{\lambda_{i}+a-b}{\lambda_{i}+a+b}=\rm{arctan}\frac{\lambda_i+a}{b}-\frac{\pi}{4}.
	\end{align}
	we take
	\begin{align}\label{51}
	\hat{u}(x)=\frac{b}{\sqrt{a^{2}+1}}u\left(\frac{(a^2+1)^{\frac{1}{4}}}{b}x\right)+\frac{a}{2b}\vert x\vert^2-\frac{n\pi}{4}
	\end{align}
	Then, we obtain $\hat{u}$ satisfies
	\begin{align}\label{52}
	\begin{cases}
	\frac{\partial \hat{u}}{\partial t}=\sum_{i=1}^{n}\rm{arctan}\xi_{i}, \quad t>0,x\in \mathbb{R}^{n},\\
	\;\; \hat{u}=\hat{u}_{0}(x), \qquad\qquad\;\; t=0,x\in \mathbb{R}^{n}.
	\end{cases}
	\end{align}
	where $\xi=(\xi_1,\xi_2,\cdots,\xi_n)$ are the eigenvalues of $D^2\hat{u}.$ \\
\indent The equation \eqref{51} implies that
		\begin{align*}
			\hat{u}_0(x)=\frac{b}{\sqrt{a^{2}+1}}u_0\left(\frac{(a^2+1)^{\frac{1}{4}}}{b}x\right)+\frac{a}{2b}\vert x\vert^2-\frac{n\pi}{4}
		\end{align*}
 it follows that
 \begin{align*}
   D^2\hat{u}_0(x)=\frac{1}{b}D^2u_0\left(\frac{(a^2+1)^{\frac{1}{4}}}{b}x\right)+\frac{a}{b}I
 \end{align*}
Since
\begin{align*}
-(b+b\eta+a)I\le D^2u_0\le(b+b\eta-a)I
\end{align*}
that is
\begin{align*}
-\left(1+\eta+\frac{a}{b}\right)I\le\frac{1}{b}D^2u_0\left(\frac{(a^2+1)^{\frac{1}{4}}}{b}x\right) \le\left(1+\eta-\frac{a}{b}\right)I
\end{align*}
 it is easy to see that
 \begin{align*}
 -(1+\eta)I\le D^2\hat{u}_0\le(1+\eta)I
 \end{align*}
\indent According to Theorem 1.1 in \cite{chau2013lagrangian}, the equation \eqref{52} has a unique longtime smooth $\hat{u}(x,t)$ which satisfies the decay estimate ${\sup \limits_ {x\in \mathbb{R}^{n}}}|D^{l}\hat{u}(x,t)|^2\le \frac{C_l}{t^{l-2}}$, where $t\in (\varepsilon_{0},+\infty)$ and $l=\{3,4,5,\cdots\}$, and so is $u$, Thus Theorem \ref{1.112} is proved.
\end{proof}
\indent To proof Theorem \ref{1.31}, we need the following lemma.
\begin{lemma}\label{5.11}
	Suppose that $u_0$ satisfies condition $A$ and $L$. Then $u(x,1)$ is a smooth solution to \eqref{15}	
\end{lemma}
\begin{proof}
	If $u_0$ satisfies condition $L$, then by Theorem \ref{1.112}, there exists a unique smooth solution $u(x,t)$ to \eqref{12111} for all $t>0$ with initial data $u_0$. It is clear that
	\begin{align*}
	u_{R}(x,t):=R^{-2}u(Rx,R^{2}t)
	\end{align*}
		is a solution to \eqref{12111} with initial value
	$$u_{R}(x,0):=R^{-2}u_{0}(Rx)=u_{0}(x).$$
	Here condition $A$ is used. Since $u_{R}(x,0)=u_{0}$, the uniqueness results in Theorem \ref{1.112} imply
	$$u(x,t)=u_{R}(x,t)$$
	for any $R>0$. Therefore $u(x,t)$ satisfies \eqref{23}, and hence $u(x,1)$ solves \eqref{15}. Namely, $u(x,1)$ is a smooth self-expanding solution.
\end{proof}
\indent We now prove Theorem \ref{1.31}.
\begin{proof}[proof of Theorem \ref{1.31}]
Let
\begin{align*}
U_0(x)=\lim_{R\to+\infty}R^{-2}u_0(Rx).
\end{align*}
It is clear that $U(x,0)$ satisfies condition $L$. Further, we obtain
$$U_0(x)=\lim_{R\to\infty}R^{-2}u_0(Rx)=\lim_{R\to\infty}R^{-2}l^{-2}u_0(Rlx)=l^{-2}U_0(lx),$$
namely, $U_0(x)$ satisfies condition $A$. Then by Lemma \ref{5.11}, we conclude that $U(x,l)$ is a self-expanding solution.\\
Define
$$u_{R}(x,t):=R^{-2}u(Rx,R^{2}t).$$
It is clear that $u_{R}(x,t)$ is a solution to \eqref{12111} with initial value $u_{R}(x,0)=R^{-2}u_0(Rx)$ satisfying condition $L$.
For any sequence $R_i\to+\infty$, we consider the limitation of $u_{R_i}(x,t)$. For $t>0$, there holds
$$D^2u_{R_i}(x,t)=D^2u(R_ix,R^{2}_it).$$
Using Theorem \ref{1.112}, we have
\begin{align*}
-(b+b\eta+a) I\le D^2u_{R_i}(x,t)\le(b+b\eta-a)I.
\end{align*}
for all $x$ and $t>0$. Moreover, according to \eqref{142} in Theorem \ref{1.112}, we get
\begin{align*}
{\sup \limits_ {x\in \mathbb{R}^{n}}}|D^{l}u_{R_i}(\cdot,t)|\le C_l\sqrt{t}^{2-l},\qquad\;  t\in(\varepsilon_{0},+\infty), l={3,4,5\cdots}.
\end{align*}
For any $m\ge1$, $l\ge0$, using \eqref{12111}, there exists a constant $C(m,l,a,b,\eta)$ such that
\begin{align*}
{\sup \limits_ {x\in \mathbb{R}^{n}}}\bigg|\frac{\partial^m}{\partial t^m}D^{l}u_{R_i}\bigg|\le C_l\sqrt{t}^{2-l-2m},\quad t\in(\varepsilon_{0},+\infty), l={3,4,5\cdots}.
\end{align*}
We observe that
\begin{align*}
\begin{aligned}
&u_{R_i}=R^{-2}_iu_0,\\
&Du_{R_i}(0,0)=R^{-1}_iDu_0(0).
\end{aligned}
\end{align*}
are both bounded. Thus $u_{R_i}(0,t)$ and $Du_{R_i}(0,t)$ are uniformly bounded with respect to $i$ for any fixed $t$. By the Arzel$\rm{\grave{a}}$-Ascoli theorem, there exists a subsequence $\{R_{k_{i}}\}$ such that $u_{R_{k_{i}}}(x,t)$ converges uniformly to a solution $\widetilde{U}(x,t)$ to \eqref{12111} in any compact subsets of $\mathbb{R}^{n}\times(0,+\infty)$, and $\widetilde{U}(x,t)$ satisfies the estimates in Theorem \ref{1.112}. Since $\frac{\partial\widetilde{U}}{\partial t}$ is  uniformly bounded for any $t>0$, $\widetilde{U}(x,t)$ converges to some function $\widetilde{U}_0(x)$ when $t\to0$. Next, we have
\begin{align*}
\begin{aligned}
\widetilde{U}_0(x)&=\lim_{t\to0}\widetilde{U}(x,t)\\
&=\lim_{t\to0}\lim_{i\to+\infty}R_i^{-2}u(R_ix,R^{2}_i t)\\
&=\lim_{i\to+\infty}\lim_{t\to0}R_i^{-2}u(R_ix,R^{2}_i t)\\
&=\lim_{i\to+\infty}R_i^{-2}u_0(R_ix)\\
&=U_0(x).
\end{aligned}
\end{align*}
According to the uniqueness results, the above limit is independent of the choice of the subsequence $\{R_{i}\}$ and
\begin{align*}
\widetilde{U}(x,t)=U(x,t).
\end{align*}
\indent In particular, letting $R=\sqrt{t}$, we have $t^{-1}u(\sqrt{t}x,t)$ converges to $U(x,1)$ smoothly and uniformly in compact subsets of $\mathbb{R}^{n}$ when $t\to+\infty.$ 	
\end{proof}

\end{document}